\documentclass[11pt,reqno]{amsart}
\usepackage{fullpage,times,graphicx,amssymb,amsmath,amsfonts,psfrag,xcolor,pdfsync}
\usepackage[colorlinks,citecolor=bbluegray,linkcolor=ddarkbrown,urlcolor=blue,breaklinks]{hyperref}
\newtheorem{theorem}{Theorem}[section]

\renewenvironment{proof}{\textbf{Proof.}}{\QED\bigskip}

\usepackage{algorithmic,algorithm}

\usepackage[square]{natbib}

\definecolor{ddarkbrown}{rgb}{0.5,0.2,0.05} \definecolor{bbluegray}{rgb}{0.05,0,0.5}

\newcommand{\BEAS}{\begin{eqnarray*}}
\newcommand{\EEAS}{\end{eqnarray*}}
\newcommand{\BEA}{\begin{eqnarray}}
\newcommand{\EEA}{\end{eqnarray}}
\newcommand{\BEQ}{\begin{equation}}
\newcommand{\EEQ}{\end{equation}}
\newcommand{\BIT}{\begin{itemize}}
\newcommand{\EIT}{\end{itemize}}
\newcommand{\BNUM}{\begin{enumerate}}
\newcommand{\ENUM}{\end{enumerate}}

\newcommand{\BA}{\begin{array}}
\newcommand{\EA}{\end{array}}

\newcommand{\ie}{{\it i.e.}}

\newcommand{\reals}{{\mathbb R}}
\newcommand{\complex}{{\mathbb C}}

\newcommand{\symm}{{\mbox{\bf S}}}  
\newcommand{\herm}{{\mbox{\bf H}}}  

\newcommand{\Rank}{\mathop{\bf Rank}}

\newcommand{\Tr}{\mathop{\bf Tr}}
\newcommand{\diag}{\mathop{\bf diag}}

\newcommand{\idm}{\mathbf{I}}

\newcommand{\QED}{~~\rule[-1pt]{6pt}{6pt}}

\renewcommand\Re{\operatorname{Re}}
\renewcommand\Im{\operatorname{Im}}

\usepackage{caption}
\usepackage{stmaryrd}

\newcommand{\C}{{\mathbb C}}

\newtheorem*{lemma*}{Lemma}

\usepackage[off]{auto-pst-pdf} 

\begin{document}
\title{Phase retrieval for imaging problems.}

\author{Fajwel Fogel}
\address{C.M.A.P., \'Ecole Polytechnique, UMR CNRS 7641}
\email{fajwel.fogel@cmap.polytechnique.fr}

\author{Ir\`ene Waldspurger}
\address{D.I., \'Ecole Normale Sup\'erieure, Paris.}
\email{waldspur@clipper.ens.fr}

\author{Alexandre d'Aspremont} 
\address{CNRS \& D.I., UMR 8548, \vskip 0ex
\'Ecole Normale Sup\'erieure, Paris, France.}
\email{aspremon@ens.fr}

\date{\today}

\begin{abstract}
We study convex relaxation algorithms for phase retrieval on imaging problems. We show that exploiting structural assumptions on the signal and the observations, such as sparsity, smoothness or positivity, can significantly speed-up convergence and improve recovery performance. We detail numerical results in molecular imaging experiments simulated using data from the Protein Data Bank (PDB).
\end{abstract}
\maketitle

\section{Introduction}
Phase retrieval seeks to reconstruct a complex signal, given a number of observations on the {\em magnitude} of linear measurements, i.e. solve
\BEQ \label{eq:ph-retrieval}
\BA{ll}
\mbox{find} & x\\
\mbox{such that} & |Ax|=b
\EA\EEQ
in the variable $x\in\complex^p$, where $A\in\reals^{n \times p}$ and $b\in\reals^n$. This problem has direct applications in X-ray and crystallography imaging, diffraction imaging, Fourier optics or microscopy for example, in problems where physical limitations mean detectors usually capture the intensity of observations but cannot recover their phase. In what follows, we will focus on problems arising in diffraction imaging, where $A$ is usually a Fourier transform, often composed with one or multiple masks (a technique sometimes called ptychography). The Fourier structure, through the FFT, considerably speeds up basic linear operations, which allows us to solve large scale convex relaxations on realistically large imaging problems. We will also observe that in many of the imaging problems we consider, the Fourier transform is very sparse, with {\em known support} (we lose the phase but observe the magnitude of Fourier coefficients), which allows us to considerably reduce the size of our convex phase retrieval relaxations.

Because the phase constraint $|Ax|=b$ is nonconvex, the phase recovery problem~\eqref{eq:ph-retrieval} is non-convex. Several greedy algorithms have been developed (see \citep{gerchberg,Fien82,Grif84,Baus02} among others), which alternate projections on the range of $A$ and on the nonconvex set of vectors $y$ such that $|y| = |A x|$. While empirical performance is often good, these algorithms can stall in local minima. A convex relaxation was introduced in \citep{Chai11} and \citep{Cand11} (who call it PhaseLift) by observing that $|A x|^2$ is a linear function of $X = x x^*$ which is a rank one Hermitian matrix, using the classical lifting argument for nonconvex quadratic programs developed in \citep{Shor87,Lova91}. The recovery of $x$ is thus expressed as a rank minimization problem over positive semidefinite Hermitian matrices $X$ satisfying some linear conditions, i.e. a matrix completion problem. This last problem has received a significant amount of attention because of its link to compressed sensing and the NETFLIX collaborative filtering problem. This minimum rank matrix completion problem is approximated by a semidefinite program which has been shown to recover $x$ for several (random) classes of observation operators~$A$ \citep{Cand11,Cand11a,Cand13}. 

On the algorithmic side, \citep{Wald12a} showed that the phase retrieval problem~\eqref{eq:ph-retrieval} can be reformulated in terms of a single phase variable, which can be read as an extension of the MAXCUT combinatorial graph partitioning problem over the unit complex torus, allowing fast algorithms designed for solving semidefinite relaxations of MAXCUT to be applied to the phase retrieval problem.

On the experimental side, phase recovery is a classical problem in Fourier optics for example \citep{Good08}, where a diffraction medium takes the place of a lens. This has a direct applications in X-ray and crystallography imaging, diffraction imaging or microscopy \citep{Harr93,Bunk07,John08,Miao08,Dier10}.

Here, we implement and study several efficient convex relaxation algorithms for phase retrieval on imaging problem instances where $A$ is based on a Fourier operator. We show in particular how structural assumptions on the signal and the observations (e.g. sparsity, smoothness, positivity, known support, oversampling, etc.) can be exploited to both speed-up convergence and improve recovery performance. While no experimental data is available from diffraction imaging problems with multiple randomly coded illuminations, we simulate numerical experiments of this type using molecular density information from the protein data bank \citep{Berm02}. Our results show in particular that the convex relaxation is stable and that in some settings, as few as two random illuminations suffice to reconstruct the image.

The paper is organized as follows. Section~\ref{s:algos} briefly recalls the structure of some key algorithms used in phase retrieval. Section~\ref{s:image} describes applications to imaging problems and how structural assumptions can sifnificantly reduce the cost of solving large-scale instances and improve recovery performance. Section~\ref{s:numres} details some numerical experiments while Section~\ref{s:uguide} describes the interface to the numerical library developed for these problems.

\subsection*{Notations}
We write $\symm_p$ (resp. $\herm_p$) the cone of symmetric (resp. Hermitian) matrices of dimension $p$ ; $\symm_p^+$ (resp. $\herm_p^+$) denotes the set of positive symmetric (resp. Hermitian) matrices. We write $\|\cdot\|_p$ the Schatten $p$-norm of a matrix, that is the $p$-norm of the vector of its eigenvalues (in particular, $\|\cdot\|_\infty$ is the spectral norm). We write $A^\dag$  the (Moore-Penrose) pseudoinverse of a matrix $A$, and $A\circ B$ the Hadamard (or componentwise) product of the matrices $A$ and $B$. For $x\in\reals^p$, we write $\diag(x)$ the matrix with diagonal $x$. When $X\in\herm_p$ however, $\diag(X)$ is the vector containing the diagonal elements of $X$. For $X\in\herm_p$, $X^*$ is the {H}ermitian transpose of $X$, with $X^*=(\bar X)^T$. Finally, we write $b^2$ the vector with components $b_i^2$, $i=1,\ldots,n$.

\section{Algorithms}\label{s:algos}

In this section, we briefly recall several basic algorithmic approaches to solve the phase retrieval problem~\eqref{eq:ph-retrieval}. Early methods were all based on extensions of an alternating projection algorithm. However, recent results showed that phase retrieval could be interpreted as a matrix completion problem similar to the NETFLIX problem, a formulation which yields both efficient convex relaxations and recovery guarantees.

\subsection{Greedy algorithms}\label{ss:greedy}
The phase retrieval problem~\eqref{eq:ph-retrieval} can be rewritten
\BEQ\label{eq:split-pb-obs}
\BA{ll}
\mbox{minimize} & \|Ax-y\|_2^2\\
\mbox{subject to} & |y|=b
\EA\EEQ
where we now optimize over both phased observations $y \in\complex^n$ and signal $x\in\complex^p$. Several greedy algorithms attempt to solve this problem using variants of alternating projections, one iteration minimizing the quadratic error (the objective of~\eqref{eq:split-pb-obs}), the next normalizing the moduli (to satisfy the constraint). We detail some of the most classical examples in the paragraphs that follow.

The algorithm~\ref{alg:GS} by \citep{gerchberg} for instance seeks to reconstruct $y=Ax$ and alternates between orthogonal projections on the range of $A$ and normalization of the magnitudes $|y|$ to match the observations $b$. The cost per iteration of this method is minimal but convergence (when it happens) is often slow. 

\begin{algorithm}[ht]
\caption{Gerchberg-Saxton.}
\begin{algorithmic} [1]
\REQUIRE An initial $y^1\in\mathbf{F}$, i.e. such that $|y^1|=b$.
\FOR{$k=1,\ldots,N-1$}
\STATE Set
\BEQ\label{alg:GS}\tag{Gerchberg-Saxton}
y_i^{k+1} = b_i\, \frac{(A A^\dag y^k)_i} {|(A A^\dag y^k)_i|},\quad i=1,\ldots,n.
\EEQ
\ENDFOR
\ENSURE $y_N \in \mathbf{F}$.
\end{algorithmic} 
\end{algorithm}

A classical ``input-output'' variant, detailed here as algorithm~\ref{alg:fienup}, introduced by \citep{Fien82} adds an extra penalization step which usually speeds up convergence and improves recovery performance when additional information is available on the support of the signal. Oversampling the Fourier transform forming $A$ in imaging  problems usually helps performance as well. Of course, in all these cases, convergence to a global optimum cannot be guaranteed but empirical recovery performance is often quite good. 

\begin{algorithm}[ht]
\caption{Fienup}
\begin{algorithmic} [1]
\REQUIRE An initial $y^1\in\mathbf{F}$, i.e. such that $|y^1|=b$, a parameter $\beta>0$.
\FOR{$k=1,\ldots,N-1$}
\STATE Set
\[
w_i= \frac{(A A^\dag y^k)_i} {|(A A^\dag y^k)_i|},\quad i=1,\ldots,n.
\]
\STATE Set
\BEQ\label{alg:fienup}\tag{Fienup}
y_i^{k+1} = y_i^{k} - \beta (y_i^{k}-b_i w_i)
\EEQ
\ENDFOR
\ENSURE $y_N \in \mathbf{F}$.
\end{algorithmic} 
\end{algorithm}

\subsection{PhaseLift: semidefinite relaxation in signal}\label{ss:phaselift}
Using a classical lifting argument by \citep{Shor87}, and writing
\[
|a_i^*x|^2=b_i^2 \Longleftrightarrow \Tr(a_ia_i^*xx^*)=b_i^2
\]
\citep{Chai11,Cand11} reformulate the phase recovery problem~\eqref{eq:ph-retrieval} as a matrix completion problem, written
\[\BA{ll}
\mbox{minimize} & \Rank(X)\\
\mbox{subject to} & \Tr(a_ia_i^*X)=b_i^2,\quad i=1,\ldots,n\\
& X \succeq 0
\EA\]
in the variable $X\in\herm_p$, where $X=xx^*$ when exact recovery occurs. This last problem can be relaxed as
\BEQ\label{eq:ph-lift}
\BA{ll}
\mbox{minimize} & \Tr(X)\\
\tag{PhaseLift}
\mbox{subject to} & \Tr(a_ia_i^*X)=b_i^2,\quad i=1,\ldots,n\\
& X \succeq 0
\EA\EEQ
which is a semidefinite program (called \ref{eq:ph-lift} by \citet{Cand11}) in the variable $X\in\herm_p$. This problem is solved in~\citep{Cand11} using first order algorithms implemented in \citep{Beck12}. This semidefinite relaxation has been shown to recover the true signal $x$ exactly for several classes of observation operators~$A$ \citep{Cand11,Cand11a,Cand13}. 

\subsection{PhaseCut: semidefinite relaxation in phase}\label{ss:sdp}
As in \citep{Wald12a} we can rewrite the phase reconstruction problem~\eqref{eq:ph-retrieval} in terms of a phase variable $u$ (such that $|u|=1$) instead of the signal~$x$. In the noiseless case, we then write the constraint $|Ax |= b$ as $Ax = \diag(b)u$, where $u \in \complex^n$ is a phase vector, satisfying $|u_i| = 1$ for $i=1,\ldots,n$, so problem~\eqref{eq:ph-retrieval} becomes
\BEQ\label{eq:split-pb}
\BA{ll}
\mbox{minimize} & \|Ax-\diag(b)u\|_2^2\\
\mbox{subject to} & |u_i|=1
\EA\EEQ
where we optimize over both phase $u\in\complex^n$ and signal $x\in\complex^p$. While the objective of this last problem is jointly convex in $(x,u)$, the phase constraint $|u_i|=1$ is not.

Now, given the phase, signal reconstruction is a simple least squares problem, i.e. given $u$ we obtain $x$ as
\BEQ
x=A^\dag\diag(b)u
\EEQ
where $A^\dag$ is the pseudo inverse of $A$. Replacing $x$ by its value in~\eqref{eq:split-pb}, the phase recovery problem becomes
\BEQ\label{eq:ph-partit}
\BA{ll}
\mbox{minimize} & u^*Mu\\
\mbox{subject to} & |u_i|=1,\quad i=1,\ldots n,
\EA\EEQ
in the variable $u\in\complex^n$, where the Hermitian matrix
\[
M=\diag(b)(\idm-AA^\dag)\diag(b)
\] 
is positive semidefinite. This problem is non-convex in the phase variable $u$. \citep{Wald12a} detailed greedy algorithm~\ref{alg:greedy} to locally optimize~\eqref{eq:ph-partit} in the phase variable.

\begin{algorithm}[ht]  
\caption{Greedy algorithm in phase.} 
\begin{algorithmic} [1]
\REQUIRE An initial $u\in\complex^n$ such that $|u_i|=1$, $i=1,\ldots,n$. An integer $N>1$.
\FOR{$k=1,\ldots,N$}
\FOR {$i = 1,\ldots n$}
\STATE Set
\BEQ\label{alg:greedy}\tag{Greedy}
u_i=\frac{-\sum_{j\neq i} M_{ji} \bar u_j}{\left| \sum_{j\neq i} M_{ji} \bar u_j\right|}
\EEQ
\ENDFOR
\ENDFOR
\ENSURE $u\in\complex^n$ such that $|u_i|=1$, $i=1,\ldots,n$.
\end{algorithmic} 
\end{algorithm}

\noindent A convex relaxation to~\eqref{eq:ph-partit} was also derived in~\citep{Wald12a} using the classical lifting argument for nonconvex quadratic programs developed in \citep{Shor87,Lova91}. This relaxation is written
\BEQ\label{eq:ph-SDP}\tag{PhaseCut}
\BA{ll}
\mbox{min.} & \Tr(UM)\\
\mbox{subject to} & \diag(U)=1,\,U\succeq 0,
\EA\EEQ
which is a semidefinite program (SDP) in the matrix $U\in\herm_n$. This problem has a structure similar to the classical MAXCUT relaxation and instances of reasonable size can be solved using specialized implementations of interior point methods designed for that problem \citep{Helm96}. Larger instances are solved in \citep{Wald12a}  using the block-coordinate descent algorithm~\ref{alg:block}.

\begin{algorithm}[ht]  
\caption{Block Coordinate Descent Algorithm for {\bf PhaseCut}.} 
\begin{algorithmic} [1]
\REQUIRE An initial $U^0=\idm_n$ and $\nu>0$ (typically small). An integer $N>1$.
\FOR{$k=1,\ldots,N$}
\STATE Pick $i\in[1,n]$.
\STATE Compute
\BEQ\label{alg:block}\tag{BlockPhaseCut}
u=U^k_{i^c,i^c}M_{i^c,i}
\quad \mbox{and}\quad
\gamma=u^* M_{i^c,i}
\EEQ
\STATE If $\gamma>0$, set
\[
U^{k+1}_{i^c,i}=U^{k+1 *}_{i,i^c}=-\sqrt{\frac{1-\nu}{\gamma}}x
\]
else
\[
U^{k+1}_{i^c,i}=U^{k+1 *}_{i,i^c}=0.
\]
\ENDFOR
\ENSURE A matrix $U\succeq 0$ with $\diag(U)=1$.
\end{algorithmic} 
\end{algorithm} 

Ultimately, algorithmic choices heavily depend on problem structure, and these will be discussed in detail in the section that follows. In particular, we will study how to exploit structural information on the signal (nonnegativity, sparse 2D FFT, etc.), to solve realistically large instances formed in diffraction imaging applications.

\section{Imaging problems}\label{s:image}


In the imaging problems we study here, various illuminations of a single object are performed through randomly coded masks, hence the matrix $A$ is usually formed using a combination of random masks and Fourier transforms, and we have significant structural information on both the signal we seek to reconstruct (regularity, etc.) and the observations (power law decay in frequency domain, etc.). Many of these additional structural hints can be used to speedup numerical operations, convergence and improve phase retrieval performance. The paragraphs that follow explore these points in more detail.

\subsection{Fourier operators} \label{ss:fourier}
In practical applications, because of the structure of the linear operator~$A$, we may often reduce numerical complexity, using the Fourier structure of $A$ to speedup the single matrix-vector product in algorithm~\ref{alg:block}. We detail the case where $A$ corresponds to a Fourier transform combined with $k$ random masks, writing $I_1,...,I_k\in\C^p$ be the illumination masks. The image by $A$ of some signal $x\in\C^p$ is then
\begin{equation*}
Ax=\begin{pmatrix}\mathcal{F}(I_1\circ x)\\\vdots\\\mathcal{F}(I_k\circ x)\end{pmatrix},
\end{equation*}
and the pseudo-inverse of $A$ also has a simple structure, with
\begin{equation*}
A^\dag\left(\begin{smallmatrix}y_1\\\vdots\\y_k\end{smallmatrix}\right)=\underset{l=1}{\overset{k}{\sum}}\,\mathcal{F}^{-1}(y_l)\circ I_l'
\end{equation*}
where $I'_l$ is the dual filter of $I_l$, which is 
\[
I'_l=\overline{I}_l/\left(\underset{s}{\sum}|I_s|^2\right).
\]
With the fast Fourier transform, computing the image of a vector by $A$ or $A^\dag$ only requires $O(kp\log(p))$ floating-point operations. For any $v\in\C^n$, $Mv=\diag(b)(\idm-AA^\dag)\diag(b)v$ may then be computed using $O(kp\log p)$ operations instead of $O(k^2p^2)$ for naive matrix-vector multiplications.

In algorithms~\ref{alg:greedy} and~\ref{alg:block}, we also need to extract quickly columns from $M$ without having to store the whole matrix. Extracting the column corresponding to index $i$ in block $l\leq k$ reduces to the computation of $AA^\dag \delta_{i,l}$ where $\delta_{i,l}\in\C^{kp}$ is the vector whose coordinates are all zero, except the $i$-th one of $l$-th block. If we write $\delta_i\in\C^p$ the Dirac in $i$, the preceding formulas yields
\begin{equation*}
AA^\dag\delta_{i,l}=\begin{pmatrix}
\delta_i\star\mathcal{F}(I_1\circ I'_l)\\
\vdots\\
\delta_i\star\mathcal{F}(I_k\circ I'_l)
\end{pmatrix}.
\end{equation*}
Convolution with $\delta_i$ is only a shift and vectors $\mathcal{F}(I_s\circ I'_l)$ may be precomputed so this operation is very fast.

\subsection{Low rank iterates} \label{ss:low-rnk}
In instances where exact recovery occurs, the solution to the semidefinite programming relaxation~\eqref{eq:ph-SDP} has rank one. It is also likely to have low rank in a neighborhood of the optimum. This means that we can often store a compressed version of the iterates $U$ in algorithm~\ref{alg:block} in the form of their low rank approximation $U=VV^*$ where $V\in\complex^{n \times k}$. Each iteration updates a single row/column of~$U$ which corresponds to a rank two update of $U$, hence updating the SVD means computing a few leading eigenvalues of the matrix $VV^*+LL^*$ where $L\in\complex^{n\times 2}$. This update can be performed using Lanczos type algorithms and has complexity $O(kn\log n)$. Compressed storage of $U$ saves memory and also speeds-up the evaluation of the vector matrix product $U_{i^c,i^c}M_{i^c,i}$ which costs $O(nk)$ given a decomposition $U_{i^c,i^c}=VV^*$, instead of $O(n^2)$ using a generic representation of the matrix~$U$.

\subsection{Bounded support} \label{ss:regularity}
In many inverse problems the signal we are seeking to reconstruct is known to be sparse in some basis and exploiting this structural information explicitly usually improves signal recovery performance. This is for example the basis of compressed sensing where $\ell_1$ penalties encourage sparsity and provide recovery guarantees when the true signal is actually sparse.

The situation is a lot simpler in some of the molecular imaging problems we are studying below since the electron density we are trying to recover is often smooth, which means that its Fourier transform will be sparse, {\em with known support}. While we lose the phase, we do observe the magnitude of the Fourier coefficients so we can rank them by magnitude. This allows us to considerably reduce the size of the SDP relaxation without losing much reconstruction fidelity, i.e. in many cases we observe that a significant fraction of the coefficients of $b$ are close to zero. From a computational point of view, sparsity in $b$ allows us to solve a truncated semidefinite relaxation~\eqref{eq:ph-SDP}. See Figure~\ref{fig:caf} for an illustration of this phenomenon on the caffeine molecule.

 \begin{figure}[htbp]
 \psfrag{Orig}[b][c]{}
 \psfrag{FT}[b][c]{}
 \psfrag{OrigFT}[b][c]{}
 \includegraphics[scale=0.8]{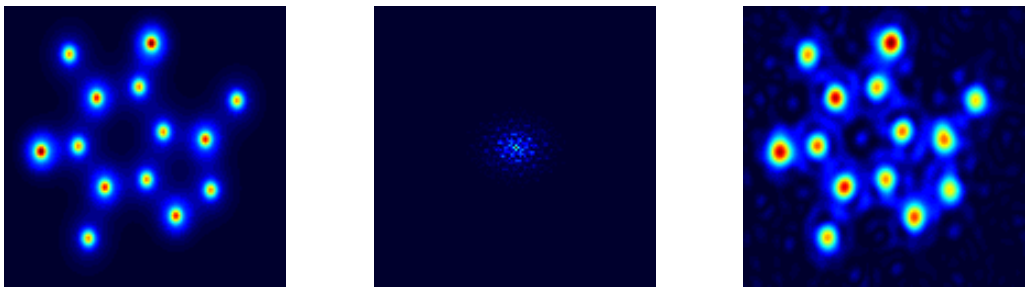}
 \caption{Electronic density for the caffeine molecule (left), its 2D FFT transform (diffraction pattern, center), the density reconstructed using 2\% of the coefficients with largest magnitude in the FFT (right). \label{fig:caf}}
 \end{figure}


Indeed, without loss of generality, we can reorder the observations $b$ such that we approximately have $ b=(b_1^T,0)^T$. Similarly, we note
\[
u=\begin{pmatrix} u_1 \\ u_2  \end{pmatrix},
\qquad
A=\begin{pmatrix} A_1 \\ A_2  \end{pmatrix},
\qquad
A^\dag=\begin{pmatrix} (A^\dag)_1 &  (A^\dag)_2  \end{pmatrix}.
\]
Using the fact that $b_2=0$, the matrix $M$ in the objective of~\eqref{eq:ph-partit} can itself be written in blocks, that is
\[
\qquad
M=\begin{pmatrix} M_1 & 0 \\ 0 & 0  \end{pmatrix}.
\]
Since $b_2=0$, any complex vector with coefficients of magnitude one can be taken for $u_2$ and the optimization problem~\eqref{eq:ph-partit} is equivalent to 
\BEQ\label{eq:ph-partit-reg}
\BA{ll}
\mbox{minimize} & u_1^*M_1u_1\\
\mbox{subject to} & |u_{1_i}|=1,\quad i=1,\ldots n,
\EA\EEQ
in the variable $u_1\in\complex^{n_1}$, where the Hermitian matrix
\[
M_1=\diag(b_1)(\idm-A_1(A^\dag)_1)\diag(b_1)
\] 
is positive semidefinite. This problem can in turn be relaxed into a PhaseCut problem which is usually considerably smaller than the original~\eqref{eq:ph-SDP} problem since $M_1$ is typically a fraction of the size of $M$.

\subsection{Real, positive densities} \label{ss:positivity}
In some cases, such as imaging experiments where a random binary mask is projected on an object for example, we know that the linear observations are formed as the Fourier transform of a {\em positive} measure. This introduces additional restrictions on the structure of these observations, which can be written as convex constraints on the phase vector. We detail two different ways of accounting for this positivity assumption.

\subsubsection{Direct nonnegativity constraints on the density}
In the case where the signal is real and nonnegative, \citep{Wald12a} show that problem~\eqref{eq:split-pb} can modified to specifically account for the fact that the signal is real, by writing it
\[
\min_{\substack{u\in\complex^n,\,|u_i|=1,\\x\in\reals^p}} ~\|Ax-\diag(b)u\|_2^2,
\]
using the operator ${\mathcal T}(\cdot)$ defined as
\BEQ\label{eq:T-op}
{\mathcal T}(Z)=
\left(\BA{cc}
\Re(Z) & -\Im(Z)\\
\Im(Z) & \Re(Z)
\EA\right)
\EEQ
we can rewrite the phase problem on real valued signal as
\[\BA{ll}
\mbox{minimize} & \left\| {\mathcal T} (A)
\left(\BA{c} x \\ 0 \EA\right) 
-
\diag\left(\BA{c} b \\ b \EA\right)
\left(\BA{c}\Re(u) \\ \Im(u)\EA\right)
\right\|_2^2\\
\mbox{subject to} & u\in\complex^n,\,|u_i|=1\\
& x\in\reals^p.
\EA\]
The optimal solution of the inner minimization problem in $x$ is given by $x=A_2^\dag B_2v$, where
\[
A_2=\left(\BA{c} \Re(A) \\ \Im(A) \EA\right),
\quad
B_2=\diag\left(\BA{c} b \\ b \EA\right),
\quad \mbox{and} \quad
v = \left(\BA{c}\Re(u) \\ \Im(u)\EA\right)
\] 
hence the problem is finally rewritten
\[\BA{ll}
\mbox{minimize} & \| (A_2 A_2^\dag B_2- B_2)v\|_2^2\\
\mbox{subject to} & v_i^2+v_{n+i}^2=1,\quad i=1,\ldots,n,
\EA\]
in the variable $v\in\reals^{2n}$. This can be relaxed as above by the following problem
\BEQ\label{eq:sdp-real}\tag{PhaseCutR}
\BA{ll}
\mbox{minimize} & \Tr(VM_2)\\
\mbox{subject to} & V_{ii}+V_{n+i,n+i}=1,\quad i=1,\ldots,n,\\
& V \succeq 0,
\EA\EEQ
which is a semidefinite program in the variable $V\in\symm_{2n}$, where 
\[
M_2=(A_2 A_2^\dag B_2- B_2)^T(A_2 A_2^\dag B_2- B_2)=B_2^T(\idm-A_2A_2^\dag)B_2.
\]
Because $x=A_2^\dag B_2v$ for real instances, we can add a nonnegativity constraint to this relaxation, using
\[
xx^T=(A_2^\dag B_2)uu^T(A_2^\dag B_2)^T
\]
and the relaxation becomes
\[\BA{ll}
\mbox{minimize} & \Tr(VM_2)\\
\mbox{subject to} & (A_2^\dag B_2)V(A_2^\dag B_2)^T \geq 0,\\
& V_{ii}+V_{n+i,n+i}=1,\quad i=1,\ldots,n,\\
& V \succeq 0,
\EA\]
which is a semidefinite program in $V\in\symm_{2n}$.

\subsubsection{Bochner's theorem and the Fourier transform of positive measures}
Another way to include nonnegativity constraints on the signal, which preserves some of the problem structure, is to use Bochner's theorem. Recall that a function $f:\reals^s \mapsto \complex$ is {\em positive semidefinite} if and only if the matrix $B$ with coefficients $B_{ij}=f(x_i-x_j)$ is Hermitian positive semidefinite for any sequence $x_i\in \reals^s$. Bochner's theorem then characterizes Fourier transforms of positive measures.

\begin{theorem} {\bf (Bochner)} 
A function $f:\reals^s \mapsto \complex$ is positive semidefinite if and only if it is the Fourier transform of a (finite) nonnegative Borel measure.
\end{theorem}
\begin{proof}
See \citep{Berg84b} for example.
\end{proof}

For simplicity, we first illustrate this in dimension one. Suppose that we observe the magnitude of the Fourier transform of a discrete nonnegative signal $x\in\reals^p$ so that
\[
\left| \mathcal{F} x \right|=b
\]
with $b\in\reals^n$. Our objective now is to reconstruct a phase vector $u\in\complex^n$ such that $|u|=1$ and
\[
\mathcal{F} x =\diag(b) u.
\]
If we define the Toeplitz matrix
\[
B_{ij}(y)=y_{|i-j|+1}, \quad 1 \leq j \leq i \leq p,
\]
so that
\[
B(y)=
\left(\BA{cccccc}
y_1 & y_2^* & & \cdots &  & y_n^*\\
y_2 & y_1 & y_2^* & & \cdots & \\
 & y_2 & y_1 & y_2^* & & \vdots\\
\vdots & & \ddots &\ddots&\ddots\\
 & \hdots & & y_2 & y_1 & y_2^*\\
y_n &  & \hdots & & y_2 & y_1
\EA\right)\]
then when $\mathcal{F} x =\diag(b) u$, Bochner's theorem states that $B(\diag(b)u)\succeq 0$ iff $x\geq 0$. The contraint $B(\diag(b)u)\succeq 0$ is a linear matrix inequality in $u$, hence is convex. 

Suppose that we observe multiple illuminations and that the $k$ masks $I_1,\ldots,I_k \in \reals^{p \times p}$ are also nonnegative (e.g. random coherent illuminations), we have
\[
Ax=\begin{pmatrix}\mathcal{F}(I_1\circ x)\\\vdots\\\mathcal{F}(I_k\circ x)\end{pmatrix},
\]
and the phase retrieval problem~\eqref{eq:ph-partit} for positive signals $x$ is now written
\[\BA{ll}
\mbox{minimize} & u^*Mu\\
\mbox{subject to} & B_j(\diag(b)u)\succeq 0, \quad j=1,\ldots,k\\
& |u_i|=1,\quad i=1,\ldots n,
\EA\]
where $B_j(y)$ is the matrix $B(y^{(j)})$, where $y^{(j)}\in\complex^p$ is the $j^{th}$ subvector of $y$ (one for each of the $k$ masks). We can then adapt the~\ref{eq:ph-SDP} relaxation to incorporate the positivity requirement. In the one dimensional case, using again the classical lifting argument in \citep{Shor87,Lova91}, it becomes 
\BEQ\label{eq:ph-SDP-pos}\tag{PhaseCut+}
\BA{ll}
\mbox{min.} & \Tr(UM)\\
\mbox{subject to} & \diag(U)=1,\,u_1=1,\\
& B_j(\diag(b)u)\succeq 0, \quad j=1,\ldots,k\\
& \left(\BA{cc}
U & u\\
u^* & 1\EA\right) \succeq 0
\EA\EEQ
in the variables $U\in\symm_n$ and $u\in\complex^n$. The phase vector $u$ is fixed up to an arbitrary global shift, and the additional constraint $u_1=1$ allows us to exclude degenerate solutions with $u=0$. Similar results apply in multiple dimensions, since the 2D Fourier transform is simply computed by applying the 1D Fourier transform first to columns then to rows.

The SDP relaxation~\ref{eq:ph-SDP-pos} cannot be solved using block coordinate descent. Without positivity constraints, the relaxation~\ref{eq:sdp-real} designed for real signals can be solved efficiently using the algorithm in~\citep{Helm96}. The constraint structure in~\ref{eq:sdp-real} means that the most expensive step at each iteration of the algorithm in~\citep{Helm96} is computing the inverse of a symmetric matrix of dimension $n$ (or less, exploiting sparsity in $b$). Sparse instances of the more complex relation~\ref{eq:ph-SDP-pos} were solved using SDPT3 \citep{Toh96} in what follows.

\section{Numerical Experiments}\label{s:numres}

We study molecular imaging problems based on electronic densities obtained from the Protein Data Bank \citep{Berm02}. From a 3D image, we obtain a 2D projection by integrating the third dimension. After normalizing these images, we simulate multiple diffraction observations for each molecule, using several random masks. Here, our masks consist of randomly generated binary filters placed before the sample, but other settings are possible \citep{Cand13}. Our vector of observations then corresponds to the magnitude of the Fourier transform of the componentwise product of the image and the filter. As in the SPSIM package~\citep{Maia13} simulating diffraction imaging experiments, random Poisson noise is added to the observations, modeling sensor and electronic noise. More specifically, the noisy intensity measurements are obtained using the following formula,
\[
I = \sqrt{\max\left\{ 0, \alpha \cdot \mathrm{Poisson}\left(\frac{|Ax|^2}{\alpha}\right) \right\}},
\] 
where $\alpha$ is the input level of noise, and $\mathrm{Poisson(\lambda)}$ is a random Poisson sample of mean $\lambda$. We ensure that all points of the electronic density are illuminated at least once by the random masks (the first mask lets all the signal go through) and call mask ``resolution" the number of pixels in a square unit of the mask. For instance masks of resolution $4\times4$ pixels in a $16\times16$ pixels image will consist of sixteen square blocks of size $4\times4$ pixels, each block being either all zeros or all ones.

We present numerical experiments on two molecules from the Protein Data Bank (PDB), namely caffeine and cocaine, with very different structure (properly projected, caffeine is mostly circular, cocaine has a star shape). Images of the caffeine and cocaine molecules at low and high resolutions are presented in Figure~\ref{fig:twoMol}. We first use ``high" $128\times 128$ pixels resolutions to evaluate the sensitivity of~\ref{eq:ph-SDP} to noise and number of masks using the fast~\ref{alg:block} algorithm (see section~\ref{s:PhaseCutBCDExp}). We then use a ``low" $16\times 16$ pixels image resolution to compare PhaseCut formulations using structural constraints, \ie~complex~\ref{eq:ph-SDP}, real~\ref{eq:sdp-real}, and~\ref{eq:ph-SDP-pos} (with positivity constraints, see section~\ref{s:phaseTransition}) on a large number of random experiments.

\begin{figure}[h]
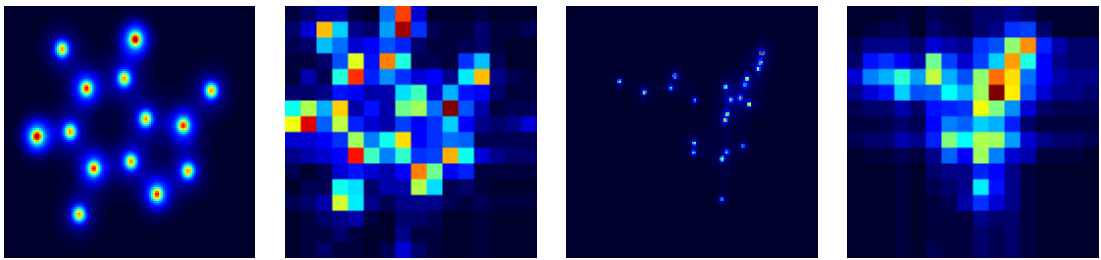

\begin{center}
\begin{tabular}{cccc}
\includegraphics[width=.33\textwidth]{figures/imageCaff128pix.eps}&
\includegraphics[width=.33\textwidth]{figures/imageCaff16pix.eps}&
\includegraphics[width=.33\textwidth]{figures/imageCoc128pix.eps}&
\includegraphics[width=.33\textwidth]{figures/imageCoc16pix.eps}\\
\end{tabular}
\end{center}
\caption{Two molecules, caffeine (left) and cocaine (right), at two resolutions: $16\times 16$ and $128\times 128$.
\label{fig:twoMol}}
\end{figure}

\begin{center}
\begin{table}
\begin{tabular}{c|c|c|c}
& Caffeine & Cocaine & Lysozyme \\\hline
Nb. atoms & 14 & 43 & 1309 \\\hline
$16\times 16$ res. & 58 \% & 40 \% & 11 \%  \\\hline
$32\times 32$ res. & 44 \% & 40 \% & 20 \%  \\\hline
$64\times 64$ res. & 15 \% & 55 \% & 14 \%  \\\hline
$128\times 128$ res. & 4 \% & 55 \% & 4 \%  \\ 
\end{tabular}
\caption{Percentage of 2D FFT coefficients required to reach $10^{-1.5}$ relative MSE, without oversampling. \label{tab:sparse1}}
\end{table}
\end{center}

\begin{center}
\begin{table}
\begin{tabular}{c|c|c|c}
& Caffeine & Cocaine & Lysozyme \\\hline
Nb. atoms & 14 & 43 & 1309 \\\hline
$16\times 16$ res. & 48 \% & 34 \% & 10 \%  \\\hline
$32\times 32$ res. & 37 \% & 35 \% & 17 \%  \\\hline
$64\times 64$ res. & 13 \% & 48 \% & 12 \%  \\\hline
$128\times 128$ res. & 4 \% & 49 \% & 4 \%  \\ 
\end{tabular}
\caption{Percentage of 2D FFT coefficients required to reach $10^{-1.5}$ relative MSE, with 2x oversampling.\label{tab:sparse2}}
\end{table}
\end{center}

\subsection{High resolution experiments using~\ref{alg:block}}\label{s:PhaseCutBCDExp}
We first compare the results obtained by the~\ref{alg:fienup} and~\ref{alg:block} algorithms while varying the number of masks and the noise level. For the PhaseCut relaxation, in order to deal with the large size of the lifted matrix, we use the low rank approximation described in \S\ref{ss:low-rnk} to store iterates and exploit sparsity in the magnitude of the  observations vector described as described in \S\ref{ss:regularity}. Tables~\ref{tab:sparse1} and~\ref{tab:sparse2} illustrate the impact of image resolution and oversampling on the fraction of coefficients required to approximately reconstruct a molecular density up to a given quality threshold, for various molecules. We observe that the sparsity of 2D FFTs increases with resolution and oversampling, but varies from one molecule to another. We then retrieve the phase vector as the first eigenvector in the final low rank approximation, then refine it with the greedy algorithms~\ref{alg:greedy} or~\ref{alg:fienup}. 

\subsubsection{Parameters}
More specifically, in the experiments that follow, the image was of size $128 \times 128$, we used a rank of two for the low rank approximation, kept the largest 1000 observations, did 5000 iterations of algorithm~\ref{alg:fienup}, and 20 cycles of algorithm~\ref{alg:block} (one cycle corresponds to optimizing once over all  rows/columns of the lifted matrix). We compared the results of the phase recovery using one to four masks, and three different levels of Poisson noise (no noise,  ``small" noise, ``large" noise). In all settings, all points of the electronic density were illuminated at least once by the random masks (the first mask lets all the signal go through). The noisy (Poisson) intensity measurements were obtained using the formula described above. Experiments were performed on a regular Linux desktop using Matlab for the greedy algorithms and a C implementation of the block coordinate algorithm for PhaseCut. Reported CPU times are in seconds.

\subsubsection{Results}
In most cases both algorithm~\ref{alg:fienup} and~\ref{alg:block} seem to converge to the (global) optimal solution, though~\ref{alg:fienup} is much faster. In some cases however, such as the experiment with two filters and no noise in Figure~\ref{fig:sdp}, initializing algorithm~\ref{alg:fienup} with the solution from~\ref{alg:block} significantly outperforms the solution obtained by algorithm~\ref{alg:fienup} alone, which appears to be stuck in a local minimum. The corresponding MSE values are listed in Table~\ref{tab:mse-img}. In Figure~\ref{fig:hist2} we plot the histogram of MSE for the noiseless case with only two illuminations, using either algorithm~\ref{alg:fienup}, or~\ref{alg:block} followed by greedy refinements, over many random illumination configurations. We observe that in many samples, algorithm~\ref{alg:fienup} gets stuck in a local optimum, while the SDP always converges to a global optimum.

 \begin{figure}[ht]
 \psfrag{a0f1}[c][c]{1 ill., $\alpha=0$}
 \psfrag{a0f2}[c][c]{2 ill., $\alpha=0$}
 \psfrag{a0f3}[c][c]{3 ill., $\alpha=0$}
 \psfrag{a0f4}[c][c]{4 ill., $\alpha=0$}
 \psfrag{a0.001f1}[c][c]{1 ill., $\alpha=10^{-3}$}
 \psfrag{a0.001f2}[c][c]{2 ill., $\alpha=10^{-3}$}
 \psfrag{a0.001f3}[c][c]{3 ill., $\alpha=10^{-3}$}
 \psfrag{a0.001f4}[c][c]{4 ill., $\alpha=10^{-3}$}
 \psfrag{a0.01f1}[c][c]{1 ill., $\alpha=10^{-2}$}
 \psfrag{a0.01f2}[c][c]{2 ill., $\alpha=10^{-2}$}
 \psfrag{a0.01f3}[c][c]{3 ill., $\alpha=10^{-2}$}
 \psfrag{a0.01f4}[c][c]{4 ill., $\alpha=10^{-2}$}
 \includegraphics[scale=1]{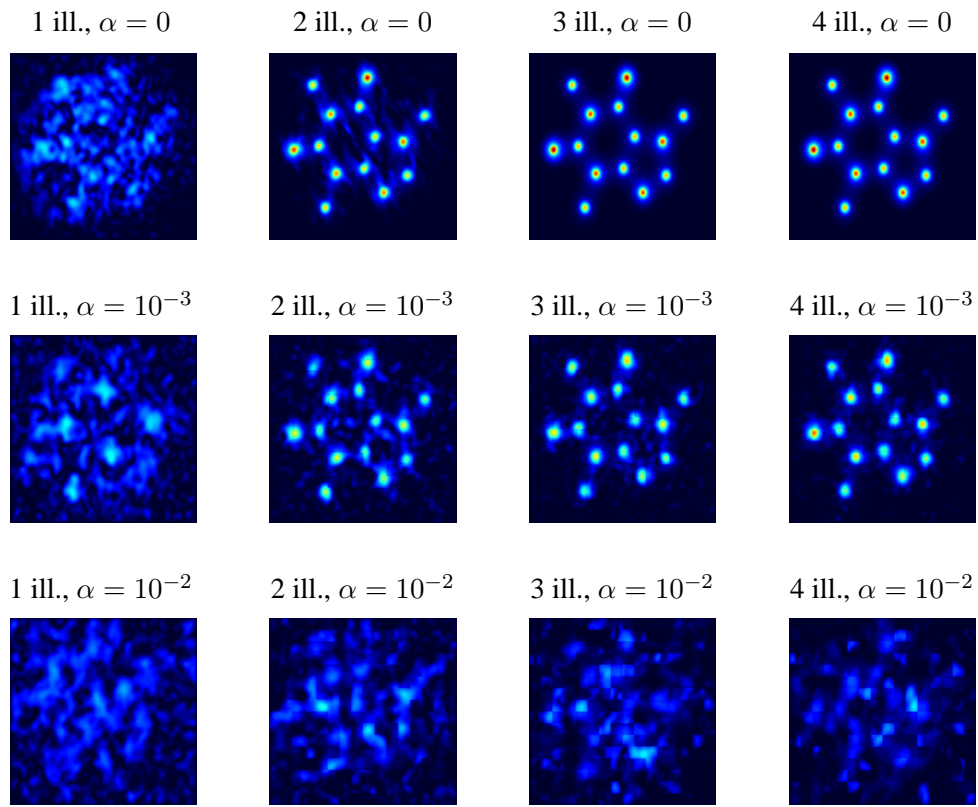}
 \caption{Solution of the semidefinite relaxation algorithm~\ref{alg:block} followed by greedy refinements, for various values of the number of  filters and noise level $\alpha$. \label{fig:sdp}}
 \end{figure}

 \begin{figure}[ht]
 \begin{center}
 \psfrag{a0f1}[c][c]{1 ill., $\alpha=0$}
 \psfrag{a0f2}[c][c]{2 ill., $\alpha=0$}
 \psfrag{a0f3}[c][c]{3 ill., $\alpha=0$}
 \psfrag{a0f4}[c][c]{4 ill., $\alpha=0$}
 \psfrag{a0.001f1}[c][c]{1 ill., $\alpha=10^{-3}$}
 \psfrag{a0.001f2}[c][c]{2 ill., $\alpha=10^{-3}$}
 \psfrag{a0.001f3}[c][c]{3 ill., $\alpha=10^{-3}$}
 \psfrag{a0.001f4}[c][c]{4 ill., $\alpha=10^{-3}$}
 \psfrag{a0.01f1}[c][c]{1 ill., $\alpha=10^{-2}$}
 \psfrag{a0.01f2}[c][c]{2 ill., $\alpha=10^{-2}$}
 \psfrag{a0.01f3}[c][c]{3 ill., $\alpha=10^{-2}$}
 \psfrag{a0.01f4}[c][c]{4 ill., $\alpha=10^{-2}$}
 \includegraphics[scale=1]{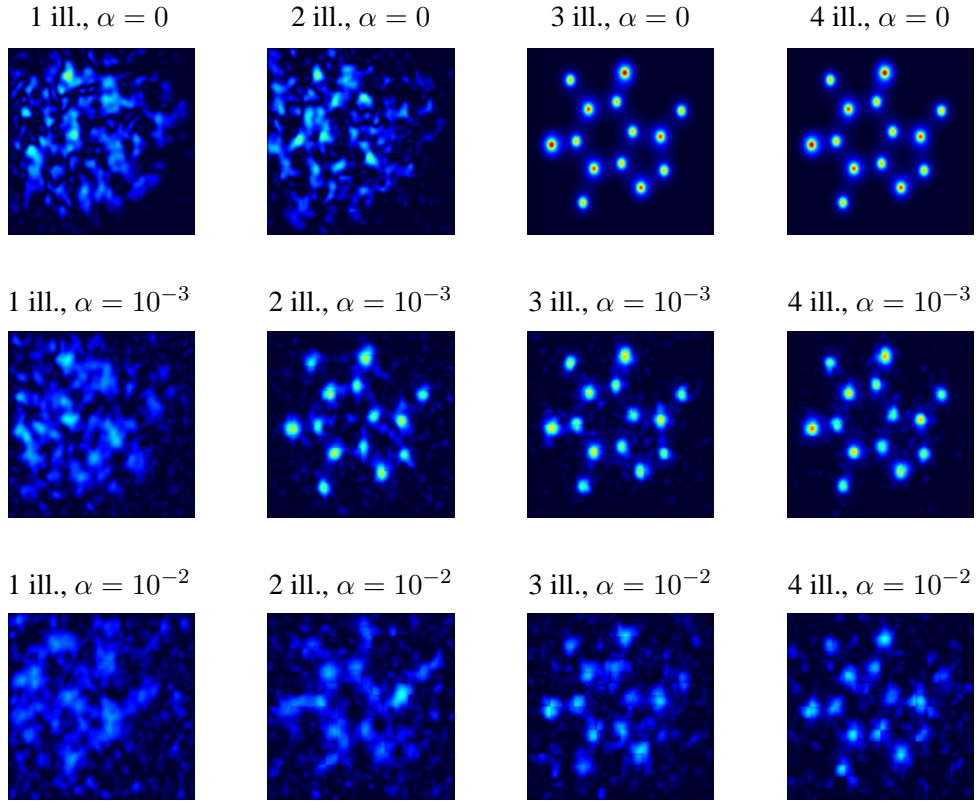}
 \end{center}
 \caption{Solution of the greedy algorithm~\ref{alg:fienup}, for various values of the number of  filters and noise level $\alpha$. \label{fig:Fienup}}
 \end{figure}

 \begin{figure}[htbp]
 \begin{center}
 \psfrag{MSE}[t][b]{MSE}
 \includegraphics[scale=0.35]{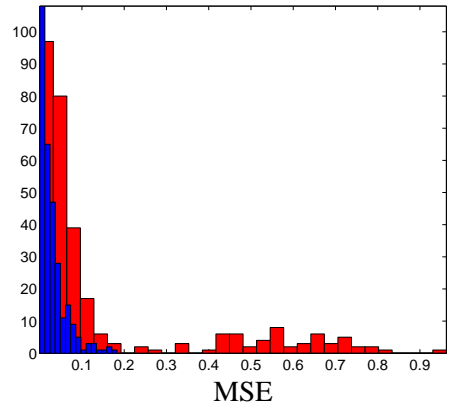}
 \caption{Histogram of MSE for the noiseless case with only two illuminations, using either algorithm~\ref{alg:fienup} (red), or~\ref{alg:block} (blue) followed by greedy refinements, over many random starts. \label{fig:hist2}}
 \end{center}
 \end{figure}

\begin{table}[h]
\begin{tabular}{c|c|c|c|c|c|c}
Nb masks&$\alpha$&SDP obj&SDP Refined obj&Fienup obj&SDP time&Fienup time\\\hline
1&0&0.410&\bf0.021&0.046&104&78\\\hline
2&0&2.157&\bf0.069&0.877&179&156\\\hline
3&0&4.036&\bf0.000&\bf0.000&256&241\\\hline
4&0&7.075&\bf0.000&\bf0.000&343&343\\\hline
1&$10^{-3}$&0.806&\bf0.691&0.704&102&67\\\hline
2&$10^{-3}$&3.300&\bf2.986&2.992&182&144\\\hline
3&$10^{-3}$&5.691&\bf5.276&5.277&263&229\\\hline
4&$10^{-3}$&8.074&\bf7.259&\bf7.259&351&308\\\hline
1&$10^{-2}$&3.299&\bf3.187&3.193&104&71\\\hline
2&$10^{-2}$&\bf7.492&8.622&8.646&182&145\\\hline
3&$10^{-2}$&\bf10.419&14.277&14.288&263&222\\\hline
4&$10^{-2}$&\bf14.139&21.472&21.444&349&305\\
\end{tabular}
\caption{Performance comparison between algorithms~\ref{alg:fienup} and~\ref{alg:block} for various values of the number of  filters and noise level $\alpha$. CPU times are in seconds. \label{tab:mse-img}}
\end{table}

\subsection{Performance of PhaseCut relaxations with respect to number of masks, noise and filter resolution}\label{s:phaseTransition}
We now compare PhaseCut formulations with structural constraints, \ie~complex~\ref{eq:ph-SDP}, real~\ref{eq:sdp-real}, and~\ref{eq:ph-SDP-pos} (with positivity constraints, see section~\ref{s:phaseTransition}) on a large number of random experiments formed using ``low" $16\times 16$ pixels image resolution.

\subsubsection{Varying the number of masks}
Masks are of resolution $1\times1$ and no noise is added.
As shown in Figures~\ref{fig:phaseTransCaff1} and~\ref{fig:phaseTransCoc1},~\ref{eq:ph-SDP},~\ref{eq:sdp-real} and~\ref{eq:ph-SDP-pos} with~\ref{alg:fienup} post processing (respectively ``SDP + Fienup HIO",  ``SDP + real + Fienup HIO", ``SDP + real + toeplitz + Fienup HIO" and ``Fienup HIO" curves on the figure) all outperform ~\ref{alg:fienup} alone. For PhaseCut, in most cases, two to three masks seem enough to exactly recover the phase. Moreover, as expected,~\ref{eq:sdp-real} performs a little bit better than \ref{eq:ph-SDP}, but surprisingly, positivity constraints of~\ref{eq:ph-SDP-pos} do not seem to improve the solution of~\ref{eq:sdp-real} in these experiments. Finally, as shown in Figures~\ref{fig:phaseTransCaff2} and~\ref{fig:phaseTransCoc2}, oversampling the Fourier transform seems to have a positive impact on the reconstruction. Results on caffeine and cocaine are very similar.

\begin{figure}[hp]
\begin{center}
\begin{tabular}{ccc}
\psfrag{nbFilters}[t][b]{Number of masks}
\psfrag{logMSE}[b][t]{MSE}
\includegraphics[scale=0.4]{./figures/logMSE_Illum_M1_n16.eps}
&~&
\psfrag{nbFilters}[t][b]{Number of masks}
\psfrag{probRec}[b][t]{Prob. of exact recovery}
\includegraphics[scale=0.4]{./figures/probRec_Illum_M1_n16.eps}
\end{tabular}
\end{center}
\caption{$16\times 16$ caffeine image. No oversampling. {\em Left:} MSE (relative to $\|b\|$) vs. number of random masks. {\em Right:} Probability of recovering molecular density ($MSE<10^{-4}$) vs. number of random masks.
\label{fig:phaseTransCaff1}}
\end{figure}

\begin{figure}[hp]
\begin{center}
\begin{tabular}{ccc}
\psfrag{nbFilters}[t][b]{Number of masks}
\psfrag{logMSE}[b][t]{MSE}
\includegraphics[scale=0.4]{./figures/logMSE_Illum_M1_n16_coc.eps}
&~&
\psfrag{nbFilters}[t][b]{Number of masks}
\psfrag{probRec}[b][t]{Prob. of exact recovery}
\includegraphics[scale=0.4]{./figures/probRec_Illum_M1_n16_coc.eps}
\end{tabular}
\end{center}
\caption{$16\times 16$ cocaine image. No oversampling.
{\em Left:} MSE (relative to $\|b\|$) vs. number of random masks. 
{\em Right:} Probability of recovering molecular density ($MSE<10^{-4}$) vs. number of random masks.
\label{fig:phaseTransCoc1}}
\end{figure}

\begin{figure}[hp]
\begin{center}
\begin{tabular}{ccc}
\psfrag{nbFilters}[t][b]{Number of masks}
\psfrag{logMSE}[b][t]{MSE}
\includegraphics[scale=0.4]{./figures/logMSE_Illum_M2_n16.eps}
&~&
\psfrag{nbFilters}[t][b]{Number of masks}
\psfrag{probRec}[b][t]{Prob. of exact recovery}
\includegraphics[scale=0.4]{./figures/probRec_Illum_M2_n16.eps}
\end{tabular}
\end{center}
\caption{$16\times 16$ caffeine image. 2x oversampling. {\em Left:} MSE vs. number of random masks. 
{\em Right:} Probability of recovering molecular density ($MSE<10^{-4}$) vs. number of random masks.
\label{fig:phaseTransCaff2}}
\end{figure}

\begin{figure}[hp]
\begin{center}
\begin{tabular}{ccc}
\psfrag{nbFilters}[t][b]{Number of masks}
\psfrag{logMSE}[b][t]{MSE}
\includegraphics[scale=0.4]{./figures/logMSE_Illum_M2_n16_coc.eps}
&~&
\psfrag{nbFilters}[t][b]{Number of masks}
\psfrag{probRec}[b][t]{Prob. of exact recovery}
\includegraphics[scale=0.4]{./figures/probRec_Illum_M2_n16_coc.eps}
\end{tabular}
\end{center}
\caption{$16\times 16$ cocaine image. 2x oversampling. {\em Left:} MSE vs. number of random masks. 
{\em Right:} Probability of recovering molecular density ($MSE<10^{-4}$) vs. number of random masks.
\label{fig:phaseTransCoc2}}
\end{figure}

\subsubsection{Varying mask resolution}

Here, two or three masks are used and no noise is added. As shown in Figures~\ref{fig:filterResTransCaff3Illum},~\ref{fig:filterResTransCoc3Illum},~\ref{fig:filterResTransCaff2Illum} and~\ref{fig:filterResTransCoc2Illum}, we can see that the MSE of reconstructed images increase with the resolution of masks. Moreover \ref{eq:sdp-real} is more robust to lower mask resolution than~\ref{eq:ph-SDP}. Finally, as expected, with more randomly masked illuminations, we can afford to lower mask resolution.

\begin{figure}[hp]
\begin{center}
\begin{tabular}{ccc}
\psfrag{filterRes}[t][b]{Mask resolution}
\psfrag{logMSE}[b][t]{MSE}
\includegraphics[scale=0.4]{./figures/logMSE_filterRes_M2_n16_3illum.eps}
&~&
\psfrag{filterRes}[t][b]{Mask resolution}
\psfrag{probRec}[b][t]{Prob. of exact recovery}
\includegraphics[scale=0.4]{./figures/probRec_filterRes_M2_n16_3illum.eps}
\end{tabular}
\end{center}
\caption{$16\times 16$ caffeine image. Mask resolution (1x1 to 8x8 pixels). {\em Left:} MSE vs. mask resolution. {(2x oversampling, no noise, 3 masks).} {\em Right:} Probability of recovering molecular density ($MSE<10^{-4}$) vs. number of random masks.
\label{fig:filterResTransCaff3Illum}}
\end{figure}

\begin{figure}[hp]
\begin{center}
\begin{tabular}{ccc}
\psfrag{filterRes}[t][b]{Mask resolution}
\psfrag{logMSE}[b][t]{MSE}
\includegraphics[scale=0.4]{./figures/logMSE_filterRes_M2_n16_3illum_coc.eps}
&~&
\psfrag{filterRes}[t][b]{Mask resolution}
\psfrag{probRec}[b][t]{Prob. of exact recovery}
\includegraphics[scale=0.4]{./figures/probRec_filterRes_M2_n16_3illum_coc.eps}
\end{tabular}
\end{center}
\caption{$16\times 16$ cocaine image. Mask resolution (1x1 to 8x8 pixels). {\em Left:} MSE vs. mask resolution. {(2x oversampling, no noise, 3 masks).} {\em Right:} Probability of recovering molecular density ($MSE<10^{-4}$) vs. number of random masks.
\label{fig:filterResTransCoc3Illum}}
\end{figure}

\begin{figure}[hp]
\begin{center}
\begin{tabular}{ccc}
\psfrag{filterRes}[t][b]{Mask resolution}
\psfrag{logMSE}[b][t]{MSE}
\includegraphics[scale=0.4]{./figures/logMSE_filterRes_M2_n16_2illum.eps}
&~&
\psfrag{filterRes}[t][b]{Mask resolution}
\psfrag{probRec}[b][t]{Prob. of exact recovery}
\includegraphics[scale=0.4]{./figures/probRec_filterRes_M2_n16_2illum.eps}
\end{tabular}
\end{center}
\caption{$16\times 16$ caffeine image. Mask resolution (1x1 to 8x8 pixels). {\em Left:} MSE vs. mask resolution. {(2x oversampling, no noise, 2 masks).} {\em Right:} Probability of recovering molecular density ($MSE<10^{-4}$) vs. number of random masks.
\label{fig:filterResTransCaff2Illum}}
\end{figure}

\begin{figure}[hp]
\begin{center}
\begin{tabular}{ccc}
\psfrag{filterRes}[t][b]{Mask resolution}
\psfrag{logMSE}[b][t]{MSE}
\includegraphics[scale=0.4]{./figures/logMSE_filterRes_M2_n16_2illum_coc.eps}
&~&
\psfrag{filterRes}[t][b]{Mask resolution}
\psfrag{probRec}[b][t]{Prob. of exact recovery}
\includegraphics[scale=0.4]{./figures/probRec_filterRes_M2_n16_2illum_coc.eps}
\end{tabular}
\end{center}
\caption{$16\times 16$ cocaine image. Mask resolution (1x1 to 8x8 pixels).
{\em Left:} MSE vs. mask resolution. {(2x oversampling, no noise, 2 masks).}
{\em Right:} Probability of recovering molecular density ($MSE<10^{-4}$) vs. number of random masks.
\label{fig:filterResTransCoc2Illum}}
\end{figure}

\subsubsection{Varying noise levels}
Here two masks are used (the minimum), with resolution $1\times1$. Poisson noise is added (parameterized by $\alpha$). As shown in Figures~\ref{fig:noiseTransCaff2Illum}, and~\ref{fig:noiseTransCoc2Illum}, we can see that~\ref{eq:ph-SDP} and~\ref{eq:sdp-real} are stable with regards to noise, i.e. we obtain a linear increase of the log MSE with respect to the log noise.

\begin{figure}[hp]
\begin{center}
\begin{tabular}{ccc}
\psfrag{noise}[t][b]{Noise}
\psfrag{logMSE}[b][t]{MSE}
\includegraphics[scale=0.4]{./figures/logMSE_noise_M2_n16_2illum.eps}
&~&
\psfrag{noise}[t][b]{Noise}
\psfrag{probRec}[b][t]{Prob. of exact recovery}
\includegraphics[scale=0.4]{./figures/probRec_noise_M2_n16_2illum.eps}
\end{tabular}
\end{center}
\caption{$16\times 16$ caffeine image. Noise. {\em Left:} MSE vs. noise level $\alpha$ {(2x oversampling, 2 masks).}
{\em Right:} Probability of recovering molecular density ($MSE<10^{-4}$) vs. number of random masks.
\label{fig:noiseTransCaff2Illum}}
\end{figure}

\begin{figure}[hp]
\begin{center}
\begin{tabular}{ccc}
\psfrag{noise}[t][b]{Noise}
\psfrag{logMSE}[b][t]{MSE}
\includegraphics[scale=0.4]{./figures/logMSE_noise_M2_n16_2illum_coc.eps}
&~&
\psfrag{noise}[t][b]{Noise}
\psfrag{probRec}[b][t]{Prob. of exact recovery}
\includegraphics[scale=0.4]{./figures/probRec_noise_M2_n16_2illum_coc.eps}
\end{tabular}
\end{center}
\caption{$16\times 16$ cocaine image. Noise. {\em Left:} MSE vs. noise level $\alpha$ {(2x oversampling, 2 masks).} {\em Right:} Probability of recovering molecular density ($MSE<10^{-4}$) vs. number of random masks.
\label{fig:noiseTransCoc2Illum}}
\end{figure}

\section{User guide}\label{s:uguide}

We provide here the instructions to artificially recover the image of a molecule from the Protein Data Bank using PhaseCutToolbox (download at \url{www.di.ens.fr/~aspremon}).
This example is entirely reproduced with comments in the script testPhaseCut.m.

\subsection{Installation}
Our toolbox works on all recent versions of MATLAB on Mac OS X, and on MATLAB versions anterior to 2008 on Linux (there might be conflicts with Arpack library for ulterior versions, when using ~\ref{alg:block}).
Installation only requires to put the toolbox folder and subdirectories on the Matlab path. Use for instance the command:

{\footnotesize\begin{verbatim}
>> addpath(genpath('MYPATH/PhaseCutToolbox'));
\end{verbatim}}

\noindent where MYPATH is the directory where you have copied the toolbox.

\subsection{Generate the diffraction pattern of a molecule}
Suppose we work with the caffeine molecule, on an image of resolution $128\times128$ pixels. We set the corresponding input variables.

{\footnotesize\begin{verbatim}
>> nameMol='caffeine.pdb';
>> N = 128 ; 
\end{verbatim}}

\noindent Now, we set the parameters of the masks. The number of masks (also called filters or illuminations) is set to 2. Moreover we set the filter resolution to 1. The filter resolution corresponds to the square root of the number of pixels in each block of the binary filter. The filter resolution must divide \texttt{\footnotesize N} (the square root of the number of pixels in the image). 

{\footnotesize\begin{verbatim}
>> filterRes = 1 ; 
>> nb_filters=2;
\end{verbatim}}

\noindent Since the filters are generated randomly, we set the seed of the uniform random generator to 1 in order to get reproducible experiments. Note that the quality of the phase retrieval may depend on the shape of the generated masks, especially when using only 2 or 3 filters.

{\footnotesize\begin{verbatim}
>> rand('seed',1);
\end{verbatim}}

\noindent Now we can generate an image, 2 masks and their corresponding diffraction patterns. We set the level of noise on the observations to zero here (\ie~no noise). $\alpha$ is the level of Poisson noise, and $\beta$ is the level of Gaussian noise.

{\footnotesize\begin{verbatim}
>> alpha=0;
>> beta=0;
\end{verbatim}}
\noindent We set the oversampling parameter for the Fourier transform to 2.
{\footnotesize\begin{verbatim}
>> OSF = 2; 
\end{verbatim}}

\noindent The total number of observations, \ie~the size of the vector $b$ is 
{\footnotesize\begin{verbatim}
>> nbObs=N*N*OSF*OSF*nb_filters;
\end{verbatim}}

\noindent Suppose that we want to use only the first largest one thousand observations in PhaseCut, we set
{\footnotesize\begin{verbatim}
>> nbObsKept=1000;
\end{verbatim}}
Note that the number of observations that is sufficient to get close to the optimal solution depends on the size of the data \texttt{\footnotesize N} and the sparsity of the vector \texttt{\footnotesize b}. From a more practical point of view, the larger \texttt{\footnotesize nbObsKept}, the more time intensive the optimization. Therefore, for a quick test we recommend setting \texttt{\footnotesize nbObsKept} to a few thousands, then increasing it if the results are not satisfying.

Finally we call the function \texttt{\footnotesize genData} which is going to generate both the image \texttt{\footnotesize x} we want to recover, \texttt{\footnotesize filters}, and observations \texttt{\footnotesize b}. \texttt{\footnotesize bs} corresponds to the thousand largest observations, \texttt{\footnotesize xs} is the image recovered with the true phase but using only \texttt{\footnotesize bs}. \texttt{\footnotesize idx\_{bs}} is the logical indicator vector of \texttt{\footnotesize bs} (\texttt{\footnotesize bs=b(idx\_{bs})}). We put \texttt{\footnotesize displayFig} to 1 in order to display the filters, the images of the molecule \texttt{\footnotesize x} and \texttt{\footnotesize xs}, as well as the diffraction patterns (with and without noise).
{\footnotesize\begin{verbatim}
>> displayFig=1;
>> [x,b,filters,bs,xs,idx_bs] = genData(nameMol, nb_filters, ...
 filterRes, N, alpha, beta, OSF, nbObsKept, displayFig);
\end{verbatim}}

\subsection{Phase Retrieval using Fienup and/or PhaseCut}

Using the data generated in the previous section, we retrieve the phase of the observations vector \texttt{\footnotesize b}.
Suppose we want to use the SDP relaxation with greedy refinement, we set
{\footnotesize\begin{verbatim}
>> method='SDPRefined';
\end{verbatim}}
The other choices for \texttt{\footnotesize method} are \texttt{\footnotesize 'Fienup'}, \texttt{\footnotesize 'Fienup HIO'} and \texttt{\footnotesize 'SDP'} (no greedy refinement). 
We set the initial (full) phase vector $u$ to the vector of ones, and the number of iterations for Fienup algorithm to 5000. The number of iterations for Fienup algorithm must be large enough so that the objective function converges to a stationary point. In most cases 5000 iterations seems to be enough.

{\footnotesize\begin{verbatim}
>> param.uInit=ones(nbObs,1);
>> param.nbIterFienup=5000;

\end{verbatim}}
We also need to choose which algorithm we want to use in order to solve the SDP relaxation. For high resolution images, we recommend to always use the block coordinate descent algorithm with a low rank approximation of the lifted matrix (BCDLR), since interior points  methods (when using SDPT3 or Mosek) and block coordinate descent without low rank approximation (BCD) become very slow when the number of observations used is over a few thousands. 
{\footnotesize\begin{verbatim}
>> param.SDPsolver='BCDLR';
\end{verbatim}}
If we had wanted to solve~\ref{eq:sdp-real} or ~\ref{eq:ph-SDP-pos} we would have set
{\footnotesize\begin{verbatim}
>> param.SDPsolver='realSDPT3';
\end{verbatim}}
or
{\footnotesize\begin{verbatim}
>> param.SDPsolver='ToepSDPT3';
\end{verbatim}}
We can now set up the parameters for the BCDLR solver.
{\footnotesize\begin{verbatim}
>> param.nbCycles=20;
>> param.r=2;
\end{verbatim}}
One cycle corresponds to optimizing over all the columns of the lifted matrix. In most cases, it seems that using \texttt{\footnotesize nbCycles} between 20 and 40 is enough to get close to the optimum, at least when refining the solution with Fienup algorithm.
\texttt{\footnotesize r} is the rank for the low rank approximation of the lifted matrix. Similarly it seems that \texttt{\footnotesize r} between 2 and 4 gives reasonable results. Note that you can check that the low rank approximation is valid by looking at the maximum ratio between the last and the first eigenvalues throughout all iterations of the BCDLR algorithm. This ratio is outputted as \texttt{\footnotesize relax.eigRatio} when calling the function \texttt{\footnotesize retrievePhase} (see below).
We finally call the function \texttt{\footnotesize retrievePhase} in order to solve the SDP relaxation with greedy refinement.
{\footnotesize\begin{verbatim}
>> data.b=b;
>> data.bs=bs;
>> data.idx_bs=idx_bs;
>> data.OSF=OSF;
>> data.filters=filters;
>> [retrievedPhase, objValues, finalObj,relax] = retrievePhase(data,method,param);
\end{verbatim}}
The function retrievePhase outputs the vector of retrieved phase as \texttt{\footnotesize retrievedPhase} and the values of the objective function at each iteration/cycle of the algorithm in \texttt{\footnotesize objValues} (add \texttt{\footnotesize .Fienup}, \texttt{\footnotesize .SDP} \texttt{\footnotesize .SDPREfined} to \texttt{\footnotesize retrievedPhase} and \texttt{\footnotesize objValues} to get the corresponding retrieved phase and objective value). If using the SDP relaxation, the vector \texttt{\footnotesize retrievedPhase} is the first eigenvector of the final lifted matrix in PhaseCut. Note that the objective value in Fienup and in the SDP relaxation do not correspond exactly since the lifted matrix may be of rank bigger than one during the iterations of the BCDLR. Therefore we also output \texttt{\footnotesize finalObj}, which is the objective value of the phase vector extracted from the lifted matrix (i.e. the vector \texttt{\footnotesize retrievedPhase}). The image can now be retrieved using the command
{\footnotesize\begin{verbatim}
>> xRetreived=pseudo_inverse_A(retrievedPhase.SDPRefined.*b,filters,M);
\end{verbatim}}
Finally you can visualize the results using the following standard Matlab commands, plotting the objective values
{\footnotesize\begin{verbatim}
>> figure(1)
>> subplot(2,1,1); 
>> title(method)
>> plot(log10(abs(objValues))); axis tight
\end{verbatim}}
and displaying images
{\footnotesize\begin{verbatim}
>> subplot(2,3,4)
>> imagesc(abs(x));axis off; 
>> subplot(2,3,5)
>> imagesc(abs(xs));axis off; 
>> subplot(2,3,6)
>> imagesc(abs(xRetreived)); axis off; 
\end{verbatim}}

\subsection{Reproduceing the experiments of the paper}
All the numerical experiments of the paper can be reproduced using the Matlab scripts included in the toolbox directory \texttt{Experiments}.
\begin{itemize}
\item \texttt{phaseTransition\_OSF1.m} (evolution of MSE with number of filters, with no oversampling of the Fourier transform, Figures~\ref{fig:phaseTransCaff1}, ~\ref{fig:phaseTransCoc1})
\item \texttt{phaseTransition\_OSF2.m} (evolution of MSE with number of filters, with oversampling of the Fourier transform  Figures~\ref{fig:phaseTransCaff2}, ~\ref{fig:phaseTransCoc2})
\item \texttt{filterResTransition.m} (evolution of MSE with filter resolution, figures~\ref{fig:filterResTransCaff2Illum},~\ref{fig:filterResTransCoc2Illum}, \ref{fig:filterResTransCaff3Illum},~\ref{fig:filterResTransCoc3Illum}).
\item \texttt{noiseTransition.m} (evolution of MSE with noise, Figures~\ref{fig:noiseTransCaff2Illum},~\ref{fig:noiseTransCoc2Illum}) 
\item \texttt{testNoiseNbIllums.m} (test noise vs number of filters, Figures~\ref{fig:sdp} and~\ref{fig:Fienup}, and table~\ref{tab:mse-img})
\item \texttt{testSeeds.m} (test different seeds to generate filters, Figure~\ref{fig:hist2})
\end{itemize}

\section*{Acknowledgments}  AA and FF would like to acknowledge support from a starting grant from the European Research Council (project SIPA).

\small{\bibliographystyle{plainnat}\bibsep 1ex
\bibliography{/Users/aspremon/Dropbox/Research/Biblio/MainPerso.bib}}
\end{document}